\documentclass[a4paper,leqno]{amsart}

\usepackage{amsmath}
\usepackage{amsthm}
\usepackage{amssymb}
\usepackage{amsfonts}
\usepackage[applemac]{inputenc}
\usepackage{mathrsfs}
\usepackage{pdfsync}
\usepackage{color} 
\usepackage{mathtools}

\def\C{{\mathbb C}}
\def\R{{\mathbb R}}
\def\N{{\mathbb N}}

\def\le{\leqslant}
\def\ge{\geqslant}

\newcommand{\re}{\mathrm{Re}}
\newcommand{\im}{\mathrm{Im}}
\newcommand{\eps}{\varepsilon}

\theoremstyle{plain}
\newtheorem{theorem}{Theorem}[section]
\newtheorem{lemma}[theorem]{Lemma}

\newtheorem{proposition}[theorem]{Proposition}

\theoremstyle{definition}
\newtheorem{definition}[theorem]{Definition}

\newtheorem{remark}[theorem]{Remark}
\newtheorem*{remark*}{Remark}

\numberwithin{equation}{section}

\begin{document}

\title[NLS with partial off-axis variations]{Regularizing nonlinear Schr\"odinger equations through partial off-axis variations}

\author[P. Antonelli]{Paolo Antonelli}

\author[J. Arbunich]{Jack Arbunich}

\author[C. Sparber]{Christof Sparber}

\address[P. Antonelli]
{Gran Sasso Science Institute, viale F. Crispi,7,67100 L'Aquilla, Italy}
\email{paolo.antonelli@gssi.infn.it}

\address[J. Arbunich]
{Department of Mathematics, Statistics, and Computer Science, M/C 249, University of Illinois at Chicago, 851 S. Morgan Street, Chicago, IL 60607, USA}
\email{jarbun2@uic.edu}

\address[C.~Sparber]
{Department of Mathematics, Statistics, and Computer Science, M/C 249, University of Illinois at Chicago, 851 S. Morgan Street, Chicago, IL 60607, USA}
\email{sparber@uic.edu}

\begin{abstract}
We study a class of focusing nonlinear Schr\"odinger-type equations derived recently by Dumas, Lannes and Szeftel  
within the mathematical description of high intensity laser beams \cite{DLS}. 
These equations incorporate the possibility of a (partial) off-axis variation of the group velocity of such laser beams through a second order 
partial differential operator acting in some, but not necessarily all, spatial directions.
We investigate the initial value problem for such models and obtain global well-posedness in $L^2$-supercritical situations, 
even in the case of only partial off-axis dependence. This provides an answer to an open problem posed in \cite{DLS}.
\end{abstract}

\date{\today}

\subjclass[2000]{35Q41, 35C20}
\keywords{Nonlinear Schr\"odinger equation, partial off-axis variation, Strichartz estimates, dispersion, finite-time blow-up, BBM equation}

\thanks{This publication is based on work supported by the NSF through grant no. DMS 1348092. 
The authors thank the anonymous referee for many helpful comments on improving our manuscript.
}
\maketitle


\section{Introduction}\label{sec:intro}

Consider the initial value problem for a general (focusing) {\it nonlinear Schr\"odinger equation} (NLS) in $d\ge 1$ spatial dimensions, i.e.,
\begin{equation}\label{NLScan}
\left\{  
	\begin{array}{lcl}
		i \partial_{t}u + \Delta u + |u|^{2\sigma }u =0, \ t\in\R, \ {\bf x} \in \R^d, \\
		u(0,{\bf x}) = u_{0}({\bf x}),
	\end{array}
\right.
\end{equation}
with $\sigma>0$, some parameter describing nonlinear effects. The NLS is a canonical model for (weakly) nonlinear wave propagation in dispersive media, cf. \cite{SuSu}. In 
particular, the {\it cubic} case ($\sigma =1$) is well-studied in the context of nonlinear laser optics, see \cite{Fi, SuSu}. The NLS
thereby describes diffractive effects which modify the propagation of slowly modulated light rays of geometrical optics over large times. In this context, the 
variable ``$t$"  should not be thought of as time, but rather as the main spatial direction of propagation of the ray. 
Solutions to \eqref{NLScan} admit several conservation laws. In particular, one finds that
\begin{equation}\label{L2}
\| u(t, \cdot) \|_{L^2}^2=\| u_0 \|_{L^2}^2,
\end{equation} 
which corresponds to the conservation of the (total) power, or intensity of the wave train.

From a mathematical point of view, it is well-known that \eqref{NLScan} is $L^2$-subcritical provided $\sigma < \frac{2}{d}$. In this regime, one can use the 
dispersive properties of the NLS to obtain global solutions $u\in C(\R_t; L^2(\R^d))$, satisfying \eqref{NLScan} in the sense of 
Duhamel's integral representation, see e.g. \cite{Cz}.
For $\frac{2}{d} \le \sigma < \frac{2}{(d-2)_+}$ one usually seeks solutions $u(t, \cdot) \in H^1(\R^d)$, in particular this includes the cubic case in dimensions $d= 2$ and $3$. 
However such a solution may not exist for all times $t\in \R$, due to the possibility of {\it finite-time blow-up}. In this case
\[
\lim_{t\to T_-} \| \nabla u(t,\cdot) \|_{L^2} = + \infty
\]
for some $T<\infty$, depending on the initial data.
A rather complete description of this phenomenon is available in the $L^2$-critical case $\sigma =\frac{2}{d}$. In particular, it is known that global solutions exist 
for intensities $\| u_0 \|_{L^2}< \| Q \|_{L^2}$, where $Q$ denotes the (stationary) ground state solution associated to \eqref{NLScan}. Above this threshold 
finite time blow-up appears and has been analyzed in a series of works, see \cite{MeRa1, MeRa2, MeRa3} and the references therein. 

From the point of view of laser physics, blow-up is usually referred to as {\it optical collapse}. However, it is known from physics 
experiments that higher order effects, neglected in the derivation of \eqref{NLScan}, can arrest such a collapse and instead yield a process called 
{\it filamentation}. The latter corresponds to a complicated interplay between diffraction, self-focusing, and defocusing mechanisms present at high intensities 
which allow the beam to propagate beyond the theoretical predicted blow-up point, see \cite{Fi}. 

In their recent mathematical study \cite{DLS}, Dumas, Lannes and Szeftel derive several new variants of the NLS from the underlying Maxwell equations of electromagnetism, 
in an effort to incorporate additional physical effects not present in \eqref{NLScan}. One of the new NLS type models derived in \cite{DLS} allows for the possibility of an 
{\it off-axis variation of the group velocity}. It takes into account the fact that self-focusing pulses usually become 
asymmetric due to variations of the group velocity within off-axis rays, a phenomenon referred to as {\it space-time focusing} in the optics literature, cf. \cite{Ro}.
To this end, the simplest mathematical model is given by 
\begin{equation}\label{NLSGVD}
		i P_{\eps} \partial_{t}u + \Delta u + |u|^{2}u =0,
\end{equation}
where $P_\eps\equiv P_\eps(\nabla)$ is a linear, second order, self-adjoint operator such that 
\[
\langle P_\eps u, u \rangle_{L^2} \gtrsim \| u \|_{L^2}^2 + \eps^2 \sum_{j=1}^k \| \omega_j \cdot \nabla u \|_{L^2}^2.
\]
Here, $\langle \cdot, \cdot \rangle_{L^2}$ denotes the usual $L^2(\R^d)$ inner product, $0< \eps\le 1$ is a 
small (dimensionless) parameter, and $\{\omega_j\}_{j=1}^k \in \R^d$, with $k\le d$, are some given (linearly independent) vectors representing the off-axis directions. 
The case $k=d$ thereby corresponds to a {\it full off-axis dependence} of the group velocity, 
whereas $k<d$ is referred to as {\it partial off-axis dependence}. In the former case, the authors of \cite{DLS} have shown that 
solutions $u(t, \cdot) \in H^1(\R^d)$ to \eqref{NLSGVD} exist for all $t\in \R$, and hence no 
finite-time blow-up occurs. The situation involving only a partial off-axis dependence, however, is much more involved and it is an open 
problem posed in \cite{DLS} to prove global well-posedness in this case. 

In this work, we shall do so and thus provide an answer to the problem posed in \cite{DLS}. To this end, we consider the following Cauchy problem:
\begin{equation}\label{NLS}
\left\{  
	\begin{array}{lcl}
		i P_{\eps} \partial_{t}u + \Delta u + |u|^{2\sigma}u =0, \ t\in\R, \ {\bf x} \in \R^d,\\
		u(0,{\bf x}) = u_{0}({\bf x}),
	\end{array}
\right.
\end{equation}
where $\sigma>0$. From now on, we shall split the spatial coordinates into ${\bf x}=(x,y) \in \R^{d-k} \times \R^{k}$ for $k \le d$, with the 
understanding that if $k=d$, we again identify $y\equiv {\bf x}\in \R^d$. 
In addition, we choose without loss of generality $ \omega_j$ to be the 
$j$th standard basis vectors in $\R^k$. Explicitly, we then have
\begin{equation}\label{P}
P_{\eps} = 1 - \eps^{2}\Delta_{y} = 1-\eps^{2}\sum_{j=1}^{k} \frac{\partial^{2}}{\partial y_j^2},\quad 0\le k\le d.
\end{equation}
With the usual summation convention, the case $k=0$ thereby corresponds to the situation with no off-axis variation, for which we will recover (as we shall see below) 
the usual $L^2$ well-posedness theory for NLS.

Mathematically, \eqref{NLS} is related to \eqref{NLScan}, in the same way the {\it Benjamin--Bona--Mahoney equation} 
is related to the celebrated {\it Korteweg--de Vries equation} for shallow, unidirectional water waves in $d=1$, see \cite{BBM}. 
The difference, when compared to our case, is that we are not confined to work in only one spatial dimension, and 
therefore can allow for a partial regularization in $k<d$ directions (a possibility which seems to have not been considered 
for BBM-type equations in higher dimensions, see \cite{GoWi}).

When comparing \eqref{NLS} to \eqref{NLScan}, one checks that, at least formally, both equations are Hamiltonian systems which (formally) conserve the 
same energy functional, i.e.,
\begin{equation}\label{CLE}
E(t) = \frac{1}{2}\|\nabla u(t, \cdot) \|^{2}_{L^2} -\frac{1}{2(\sigma +1)}\| u(t, \cdot) \|^{2\sigma +2}_{L^{2\sigma +2}} = E(0).
\end{equation}
However, instead of the usual $L^2$ conservation law \eqref{L2}, one finds
\begin{equation} \label{CLM}
\| P^{1/2}_{\eps}u(t, \cdot) \|^2_{L^2} =\| P^{1/2}_{\eps}u_{0}\|^2_{L^2} 
\end{equation}
in the case of \eqref{NLS}. Here, and in the following, $P^{1/2}_{\eps}$ is the pseudo-differential operator corresponding to the Fourier symbol 
\begin{equation}\label{sym}
\widehat P^{1/2}_\eps(\eta)= (1+ \eps^2 |\eta|^2)^{1/2} \quad \text{for $\eta \in \R^k$.}
\end{equation}
The identity \eqref{CLM} corresponds to a conservation law for (the square of) the mixed $L^{2}(\R_{x}^{d-k}; H^{1}(\R_{y}^k))$-norm of $u$, whenever $\eps >0$. 
In order to understand the influence of partial off-axis variations, it is therefore natural to set up a well-posedness theory in this mixed Sobolev-type space. \\

With this in mind, we can now state the main results of this work.

\begin{theorem}[Partial off-axis variation; subcritical case]\label{thm:sub}
Let $d> k\ge 0$ and
\begin{itemize}
\item either $k\le 2$ and $0 \le \sigma < \frac{2}{d-k}$, 
\item or $k>2$ and $0 \le \sigma \le \frac{2}{d-2}$.
\end{itemize}
Then for any $u_0\in L^{2}(\R_{x}^{d-k}; H^{1}(\R_{y}^k))$ there exists a unique global-in-time solution 
$u\in C(\R_t;L^{2}(\R_x^{d-k}; H^{1}(\R_{y}^k)))$ to \eqref{NLS}, depending 
continuously on the initial data and satisfying the conservation law \eqref{CLM} for all $t\in \R$. 
\end{theorem}

In the result above, we have to exclude the choice $k\le 2$ and $\sigma=\frac{2}{d-k}$, which corresponds to a critical case that needs to be dealt with separately (see below). 
Regardless of that, we see that as soon as $k>0$, i.e., as soon as some partial off-axis variation is present, 
we can allow for $L^2$-{\it supercritical} powers $\sigma> \frac{2}{d}$ and still retain global-in-time solutions $u$. In other words, no finite time blow-up appears  
in the case of partial off-axis variations, and we can even allow for initial data $u_0$ in a space slightly larger than $H^1(\R^d)$.  

We now turn to the case of partial off-axis dispersion with critical nonlinearity, for which we can prove an analogue of the well-posedness results given in \cite{CW}. 
Note that for $k=0$ (no off-axis variation) we recover the usual $L^2$-critical case $\sigma =\frac{2}{d}$.

\begin{theorem}[Partial off-axis variation; critical case] \label{thm:crit}
Let $0\le k\le 2$, and $\sigma=\frac{2}{d-k}$. Then for any $u_0\in L^{2}(\R_{x}^{d-k}; H^{1}(\R_{y}^k))$ 
there exist times $0<T_{\rm max}, T_{\rm min}\le\infty$ and a unique maximal solution 
$u\in C((-T_{\rm min}, T_{\rm max});L^{2}(\R_{x}^{d-k}; H^{1}(\R_{y}^k)))$, satisfying \eqref{CLM} for all $t\in (-T_{\rm min}, T_{\rm max})$.
In addition, we have the following blow-up alternative: $T_{\rm max}<\infty$ if and only if
\begin{equation*}
\|u\|_{L^{\frac{2(d-k+2)}{d-k}}\big([0, T_{\rm max}) \times \R_x^{d-k} ; H^{\frac{2}{d-k+2}}(\R_y^k)\big)}=\infty,
\end{equation*}
and analogously for $T_{\rm min}$. Finally, if the $L^{2}(\R_{x}^{d-k}; H^{1}(\R_{y}^k))$-norm of the initial datum is sufficiently small, 
then the solution $u$ exists for all $t\in \R$.
\end{theorem}

For completeness, we shall also state a result in the case of full off-axis variation. Note when $k=d$, the mixed Sobolev space above simply becomes $H^1(\R^d)$. 

\begin{theorem}[Full off-axis variation] \label{thm:full}
Let $k=d$ and $0\le \sigma \le \frac{2}{(d-2)_+}$. Then for any $u_0\in H^{1}(\R^d)$ there exists a unique global-in-time solution 
$u\in C( \R_t; H^{1}(\R^d) )$ to \eqref{NLS}, depending 
continuously on the initial data and satisfying the conservation laws \eqref{CLE} and \eqref{CLM} for all $t\in \R$.
\end{theorem}

This is a slight generalization of the result given in \cite{DLS}, where only the cubic case is treated. 
Note that we can allow for $\sigma = \frac{2}{(d-2)_+}$, i.e., the $H^1$-critical power, in contrast to the usual theory of NLS without off-axis variation, cf. \cite{KM}.

In order to prove all of these theorems, we shall employ the following change of unknown
\begin{equation}\label{eq:trans}
v(t,{\bf x} ):= P_{\eps}^{1/2}u(t, {\bf x}), 
\end{equation}
and rewrite the Cauchy problem \eqref{NLS} in the form
\begin{equation}\label{PNLS}
\left\{  
	\begin{array}{lcl}
		i \partial_{t}v + P_{\eps}^{-1}\Delta v + P_{\eps}^{-1/2}\big(|P_{\eps}^{-1/2}v|^{2\sigma}P_{\eps}^{-1/2}v\big) =0, \ t\in \R, \ {\bf x} \in \R^d,\\
	        v(0,{\bf x}) = P_{\eps}^{1/2}u_{0}({\bf x})\equiv v_0({\bf x}).
	\end{array}
\right.
\end{equation}
Instead of \eqref{CLM}, this new equation conserves
\[
\| v(t, \cdot) \|^{2}_{L^{2}} =\| P^{1/2}_{\eps}u(t, \cdot) \|^2_{L^2} =\| P^{1/2}_{\eps}u_{0}\|^2_{L^2} = \| v_{0} \|^{2}_{L^{2}},
\]
i.e., the usual $L^2$-conservation law. We therefore aim to set-up an $L^2$-based well-posedness theory for \eqref{PNLS}, written in Duhamel's form, i.e.
\[
v(t) = e^{itP^{-1}_{\eps}\Delta}v_{0} + i\int_{0}^{t} e^{i(t-s)P^{-1}_{\eps}\Delta}P_{\eps}^{-1/2}(|P_{\eps}^{-1/2}v|^{2}P_{\eps}^{-1/2}v)(s) \;ds.
\]
The advantage of working with $v$ instead of $u$ lies in the fact that it allows us to exploit the regularizing properties of the operator $P_{\eps}^{-1/2}$ acting on the nonlinearity. 
Roughly speaking, the action of $P_{\eps}^{-1/2}$ allows us to gain a derivative in $y\in \R^k$. 
However, we also note that the linear semi-group 
\begin{equation}\label{S}
S_\eps(t)=e^{itP_{\eps}^{-1}\Delta}
\end{equation} 
is no longer dispersive in the same way as the usual Schr\"odinger group $S_0(t)=e^{it \Delta}$.
Indeed, we can only expect ``nice" dispersive properties in the spatial 
directions $x\in \R^{d-k}$, where $P_\eps$ does not act, which will play an important role in the derivation of suitable Strichartz estimates (see below). 
It has been proved in \cite{Car} 
that in the case of full off-axis dependence, $S_\eps(t)$ does not admit any Strichartz estimates.
Note that this issue is not simply an artifact of our change of unknown $u \mapsto v$, since $S_\eps(t)$ also describes the dispersive properties of (the linear part of) 
the original equation for $u$, as 
can be seen by applying $P_\eps^{-1}$ to the first line of \eqref{NLS}. This issue has already been noticed in \cite{DLS}, but the change of unknown $u\mapsto v$, which 
allows us to treat the partial off-axis variation, is a novel idea of the present paper.

We also want to mention that the sign of the nonlinearity (which is focusing) does not play a role in the proofs given below, and hence all of our results 
also remain true in the defocusing case. 
\\

This paper is organized as follows: In the next section we shall introduce some notations and definitions.
Then in Section \ref{sec:SE}, we shall study the dispersive properties of $S_\eps(t)$ and derive appropriate Strichartz estimates in the case of 
partial off-axis dispersion. These will then be used in Section \ref{sec:CP} to prove global well-posedness of \eqref{PNLS} in the subcritical case. The critical case, and the case of full off-axis dispersion,
will be treated in Section \ref{sec:crit}.


\section{Basic notations and definitions} \label{sec:not} 

As mentioned in the Introduction, we shall denote ${\bf x}=(x, y)\in\R^{d-k}\times\R^k$ with the understanding that if either $k=0$ (no off-axis variation) or if $k=d$ (full off-axis variation), 
the variable $y$ does not appear. We will often use mixed Lebesgue spaces such as $L^{p}(\R_{x}^{d-k};L^q(\R_y^k))$, which will be shortly denoted by $L^p_xL^q_y$. 
These spaces are equipped with the following norms:
\begin{equation*}
\|f\|_{L^p_xL^q_y}:=\left(\int_{\R^{d-k}}\left(\int_{\R^k}|f(x, y)|^q\,dy\right)^{\frac{p}{q}}\,dx\right)^{\frac1p}.
\end{equation*}

We denote the usual Fourier transform of a function $f=f(x,y)$ as
\[
(\mathcal F f)(\xi,\eta)\equiv \widehat{f}(\xi,\eta) = \frac{1}{(2\pi)^{d/2}} \iint_{\R^d} f(x,y) e^{-i  (x\cdot \xi + y\cdot \eta) } \, dx \, dy,
\]
whereas the partial Fourier transform with respect to the $y$-variable only will be denoted by
\[
(\mathcal{F}_{y \to \eta} f)(x,\eta)\equiv \tilde{f}(x,\eta) = \frac{1}{(2\pi)^{k/2}}\int_{\R^k} f(x,y) e^{-i y\cdot \eta} \, dy.
\]
Analogously, we denote the partial Fourier transform in $x$ by $\mathcal{F}_{x \to \xi}$.

By recalling the (family of) differential operators $P_\eps=1-\eps^2\Delta_y$, defined in \eqref{P} with $0< \eps \le 1$, 
we shall introduce the class of mixed Sobolev-type spaces $L^{p}(\R_{x}^{d-k}; H^{s}(\R_{y}^k))$ of order $s\in \R$, via the following norm
\begin{equation*}
\|f\|_{L^{p}_{x} H^{s}_{y}}:=\big \|P_1^{s/2}f\big \|_{L^{p}_{x}L^{2}_{y}}\equiv \| (1+| \eta|^2)^{s/2} \tilde f \|_{L^{p}_{x}L^{2}_{\eta}}.
\end{equation*}
Obviously, the Fourier symbol corresponding to $P_1^{1/2}$ is nothing but the well-known Japanese bracket $\langle \eta \rangle = (1+| \eta|^2)^{1/2}$ used in the 
definition of $H^s$. 
Incorporating the small parameter $0<\eps\le 1$ comes at the expense of some (possibly) $\eps$-dependent constants: Indeed, for $s\ge 0$, we have
\begin{equation}\label{act1}
\eps^{s}\|f \|_{H^s} \le \|P_{\eps}^{s/2} f \|_{L^2} \le \|f \|_{H^s},
\end{equation}
as well as
\begin{equation}\label{act2}
\|f \|_{H^{-s}} \le \|P_{\eps}^{-s/2} f \|_{L^2} \le \eps^{-s}\|f \|_{H^{-s}}.
\end{equation}
From now on, we shall write $a \lesssim b$ whenever there exists a universal constant $C>0$, independent of $\eps$, 
such that $a\le Cb$. In general this constant $C$ may change from inequality to inequality.

Furthermore, for any time interval $I\subset \R$ we will also make use of the mixed space-time spaces 
$L^{q}( I_t , L^{p}(\R_x^{d-k};H^{s}(\R^{k}_y)))$ briefly denoted by $L^q_tL^p_xH^{s}_y(I)$, or simply $L^q_tL^p_xH^{s}_y$, whenever the time interval is clear. 
These spaces are equipped with the norm
\begin{equation*}
\|F\|_{L^q_tL^p_xH^s_y}:=\left(\int_{I}\|F(t)\|_{L^p_xH^s_y}^q\,dt\right)^{\frac1q}.
\end{equation*}
Associated with these spaces is the following notion of Strichartz admissibility.
\begin{definition}\label{def:adm}
Let $d>k\ge 0$ be given. We say that the pair $(q, r)$ is {\it admissible} if 
$2\le r\le\infty, 2\le q\le\infty$, and
\[
\frac2q=(d-k)\left(\frac12-\frac1r \right)=:\delta(r)
\]
where we omit the endpoint case, i.e., $(q,r) \ne (2,\frac{2(d-k)}{(d-k-2)_{+}})$ for $d-k\ge 2 $.
\end{definition}
Clearly, if $k=0$, this is just the usual admissibility condition for nonendpoint Strichartz pairs corresponding to the Schr\"odinger group $S_0(t)=e^{it \Delta}$ acting on $\R^d$.


\section{Dispersive properties with partial off-axis variation}\label{sec:SE}

In this section, we shall derive Strichartz estimates associated to $S_\eps(t)=e^{itP_{\eps}^{-1}\Delta}$ in the case of partial off-axis variation, i.e. $d>k$. To this end we 
first derive a set of basic dispersion estimates associated to this linear propagator.

\subsection{Dispersion estimate for $\bf{S_\eps(t)}$} 

Recall the notation $\delta(r)\ge0$ introduced in Definition \ref{def:adm}. Then we have the following.

\begin{proposition}\label{prop:DisE}
Let $r \in [2,\infty]$, and $t \neq 0$. Then, for any $\eps>0$,  
the group of $L^2$-unitary operators $S_\eps(t)=e^{itP_{\eps}^{-1}\Delta}$ continuously maps 
\[L^{r'}(\R_x^{d-k};H^{\delta(r)}(\R^{k}_y)) \to L^{r}(\R_x^{d-k};H^{-\delta(r)}(\R^{k}_y)), \quad \text{for} \ \ \frac{1}{r} + \frac{1}{r^{\prime}} = 1,\] 
and it holds that
\begin{equation}\label{DisE1}
 \| S_\eps(t) f  \|_{L^{r}_{x}H^{-\delta(r)}_{y}} \le |4 \pi t|^{-\delta(r)}  \| f  \|_{L^{r'}_{x}H^{\delta(r)}_{y}}.
\end{equation}
\end{proposition}

\begin{proof}
The estimate \eqref{DisE1} will in itself be a consequence of the following inequality, which is more directly 
linked to the explicit form of our propagator $S_\eps(t)=e^{itP_{\eps}^{-1}\Delta}$:
\begin{equation}\label{DisE}
 \| S_\eps(t) f  \|_{L^{r}_{x}L^{2}_{y}} \le |4 \pi t|^{-\delta(r)}  \| P_{\eps}^{\delta(r)}f  \|_{L^{r'}_{x}L^{2}_{y}}.
\end{equation}
Indeed, if we replace $f$ by $P_{\eps}^{-\frac{\delta(r)}{2}}f$ in \eqref{DisE} and keep in mind the basic estimates \eqref{act2} and \eqref{act1}, 
we obtain \eqref{DisE1} through the string of inequalities
\begin{align*}
\| S_\eps(t) f  \|_{L^{r}_{x}H^{-\delta(r)}_{y}}  &\, \le \| S_\eps(t) P_{\eps}^{-\frac{\delta(r)}{2}}f  \|_{L^{r}_{x}L^{2}_{y}}
\le |4 \pi t|^{-\delta(r)}  \| P_{\eps}^{\frac{\delta(r)}{2}}f \|_{L^{r'}_{x}L^{2}_{y}} \\
& \, \le |4 \pi t|^{-\delta(r)}  \| f  \|_{L^{r'}_{x}H^{\delta(r)}_{y}},
\end{align*}
which also ensures the continuity of $S_\eps(t)$. We also point out that there are no $\eps$-dependent constants involved in any of these inequalities.

In order to prove \eqref{DisE}, we first note that by density, it is enough to show this for $f\in \mathcal S(\R^d)$, the space of smooth and rapidly decaying functions. 
Moreover, we shall argue by duality and rather prove that for $f, g\in \mathcal S(\R^d)$,
\begin{equation}\label{dualest}
|\langle S_\eps(t)f , g \rangle_{L^{2}} |\le |4 \pi t|^{-\delta(r)}\|P_{\eps}^{\delta(r)}f\|_{L^{r'}_xL^{2}_y}\|g\|_{L^{r'}_xL^2_y}.
\end{equation}
In the trivial case $r=2$, $\delta(r)=0$, this estimate directly follows by Cauchy--Schwarz and the fact that $S_\eps(t)$ is unitary on $L^2$:
\begin{equation}\label{eq:103}
|\langle S_\eps(t)f , g \rangle_{L^{2}}| \le \|S_\eps(t)f\|_{L^{2}}\|g\|_{L^{2}} = \|f\|_{L^{2}}\|g\|_{L^{2}}.
\end{equation} 

Next, we treat the case $r=\infty$, $\delta(r)=\frac{d-k}{2}$, i.e., we want to show that for $f, g\in \mathcal S(\R^d)$ it holds that
\begin{equation}\label{eq:102}
|\langle S_\eps(t)f , g \rangle_{L^{2}}| \le |4 \pi t|^{-\frac{d-k}{2}}\|P_{\eps}^{\frac{d-k}{2}}f\|_{L^{1}_{x}L^{2}_{y}}\|g\|_{L^{1}_{x}L^2_{y}}.
\end{equation} 
To this end, we use Plancherel's identity to write
\begin{align*}
\langle S_\eps(t)f , g \rangle_{L^{2}} &= \big \langle\, \widehat{(S_\eps(t)f)} , \widehat{g} \, \big \rangle_{L^{2}} = 
\iint\limits_{\R^{d-k}\times \R^{k}}e^{-\frac{i(|\eta|^2+|\xi|^{2})t}{1+\eps^2|\eta|^2}}\widehat{f}(\xi,\eta)\overline{\widehat{g}(\xi,\eta)}\,d\xi \, d\eta\\
&=\int\limits_{\R^{k}}e^{-\frac{i|\eta|^2t}{1+\eps^2|\eta|^2}}\Big(\int\limits_{\R^{d-k}}e^{-\frac{i|\xi|^2t}{1+\eps^2|\eta|^2}}\widehat{f}(\xi,\eta)\overline{\widehat{g}(\xi,\eta)}\,d\xi \Big)d\eta.
\end{align*}
Here, we first compute the inner integral by writing out the partial Fourier transform in $\xi$ on $\widehat{g}$ to obtain
\begin{equation}\label{I1}
\begin{aligned}
\int\limits_{\R^{d-k}}e^{-\frac{i|\xi|^2t}{1+\eps^2|\eta|^2}}&\widehat{f}(\xi,\eta)\overline{\widehat{g}(\xi,\eta)}\,d\xi = \\&= 
\frac{1}{(2\pi)^{\frac{d-k}{2}}}\int\limits_{\R^{d-k}} e^{-\frac{i|\xi|^2t}{1+\eps^2|\eta|^2}}\widehat{f}(\xi,\eta)  \int\limits_{\R^{d-k}} e^{ix\cdot\xi}\overline{\tilde{g}(x,\eta)} \;dxd\xi \\
&= \int\limits_{\R^{d-k}} \overline{\tilde{g}(x,\eta)} \bigg( \frac{1}{(2\pi)^{\frac{d-k}{2}}} \int\limits_{\R^{d-k}} e^{ix\cdot\xi}e^{-\frac{i|\xi|^2t}{1+\eps^2|\eta|^2}}\widehat{f}(\xi,\eta) \;d\xi\bigg)dx\\
&= \int\limits_{\R^{d-k}} \overline{\tilde{g}(x,\eta)} \mathcal{F}^{-1}_{\xi \to x}\Big( e^{-\frac{i|\cdot|^2t}{1+\eps^2|\eta|^2}}\widehat{f}(\cdot,\eta) \Big)(x)dx,
\end{aligned}
\end{equation}
where we have used Fubini's theorem to change the order of integration.
We now recall that for $a\in \R $,
$$
\mathcal{F}^{-1}_{\xi \to x}\Big(e^{-\frac{i|\cdot|^{2}t}{a}}\Big)(z)= \Big(\frac{a}{2it}\Big)^{\frac{d-k}{2}}e^{\frac{ia|z|^{2}}{4t}}.
$$
By setting $a = 1 +\eps^{2}|\eta|^{2}$, we can express the integrand in the last line of \eqref{I1} as
\begin{align}\label{I2}
\mathcal{F}^{-1}_{\xi \to x}\Big( e^{-\frac{i|\cdot|^2t}{1+\eps^2|\eta|^2}}\widehat{f}(\cdot,\eta) \Big)(x)&
= \frac{1}{(2\pi)^{\frac{d-k}{2}}}\Big(\mathcal{F}_{\xi \to x}^{-1}\Big(e^{-\frac{i|\cdot|^2t}{1+\eps^2|\eta|^2}}\Big)\ast \tilde{f}(\cdot,\eta)  \Big)(x) \nonumber \\ 
&= \Big(\frac{1 +\eps^{2}|\eta|^{2}}{4 \pi it}\Big)^{\frac{d-k}{2}}\int \limits_{\R^{d-k}} e^{\frac{i(1 +\eps^{2}|\eta|^{2})|x-z|^{2}}{4t}} \tilde{f}(z,\eta)\;dz.
\end{align}
Now it is clear by \eqref{I1} and \eqref{I2} that
\begin{equation*}
\begin{aligned}
&(4\pi i t)^{\frac{d-k}{2}} \langle S_\eps(t)f , g \rangle_{L^{2}} \\
&= \iiint\limits_{\R^{k}\times (\R^{d-k})^{2}}(1 +\eps^{2}|\eta|^{2})^{\frac{d-k}{2}}e^{-\frac{i|\eta|^2t}{1+\eps^2|\eta|^2}}   e^{\frac{i(1 +\eps^{2}|\eta|^{2})|x-z|^{2}}{4t}} \tilde{f}(z,\eta)\overline{\tilde{g}(x,\eta)}\;dz \, dx \, d\eta.
\end{aligned}
\end{equation*}
This implies the following estimate:
\begin{align*}
|\langle S_\eps(t)f , g \rangle_{L^{2}}| \le |4\pi t|^{-\frac{d-k}{2}}\int\limits_{(\R^{d-k})^{2}}\Big(\int \limits_{\R^{k}}(1 +\eps^{2}|\eta|^{2})^{\frac{d-k}{2}} |\tilde{f}(z,\eta)| |\tilde{g}(x,\eta)|\;d\eta \Big)  dx \, dz.
\end{align*}
A Cauchy--Schwarz inequality in $\eta$, followed by Plancherel's identity, then gives
\begin{equation*}
\begin{aligned}
|\langle S_\eps(t)f , g \rangle_{L^{2}}| &\le |4\pi t|^{-\frac{d-k}{2}}\iint\limits_{(\R^{d-k})^{2}}\Big(\|\widehat{P}^{\frac{d-k}{2}}_{\eps}\tilde f(z,\cdot)\|_{L^2_\eta}\|\tilde g(x,\cdot)\|_{L^2_\eta} \Big)  dx \, dz\\
&\le |4 \pi t|^{-\frac{d-k}{2}}\|P^{\frac{d-k}{2}}_{\eps} f\|_{L^1_xL^2_y}\|g\|_{L^1_xL^2_y},
\end{aligned}
\end{equation*}
which is the desired estimate \eqref{eq:102}. 
Notice that by replacing $f\mapsto | 4\pi t|^{\frac{(d-k)}{2}}P_\eps^{-\frac{d-k}{2}}f$ in \eqref{eq:102}, 
this yields that the operator
\begin{equation}\label{eq:neu}
|4\pi t|^{\frac{d-k}{2}}S_\eps(t)P_\eps^{-\frac{d-k}{2}} \, : \, L^1_xL^2_y \to L^\infty_xL^2_y \ \ \ \text{is bounded,}
\end{equation}
with norm
\begin{equation*}
\||4\pi t|^{\frac{d-k}{2}}S_\eps(t)P_\eps^{-\frac{d-k}{2}}\| \le 1.
\end{equation*}
We have thus proved \eqref{DisE} in the two endpoint cases $r=2$ and $r=\infty$. The intermediate cases of \eqref{DisE} then follow by Stein's interpolation theorem \cite{St,StWe}. 

To this end,  we consider, for any $z\in \Omega:=\{0\le\re z\le1\}\subset \C$, the family of interpolating operators $T_z$ given by
\begin{equation*}
\mathcal F(T_zf)(\xi, \eta)=|4 \pi t|^{\frac{(d-k)z}{2}}(1+\eps^2|\eta|^2)^{-\frac{d-k}{2}z}e^{-it(1+\eps^2|\eta|^2)^{-1}(|\xi|^2+|\eta|^2)}\widehat f(\xi, \eta).
\end{equation*}
Clearly, for $z=0$, this is nothing but the Fourier transform of $S_\eps(t)$, which we know to be bounded $L^2\to L^2$ in view of \eqref{eq:103}. 
For $z=1$, we obtain the second endpoint case given by \eqref{eq:neu}.
In addition, it is straightforward to check that $\{T_z\}_{z\in \Omega}$ is an admissible family of linear operators satisfying the hypotheses of Theorem V.4.1 in \cite{StWe}. 
The theorem then requires us to bound $T_z$ at the edges of the strip $\Omega$: 

For $\mu  \in \R$, the following estimate for $z = 0 + i\mu$ uses \eqref{eq:103} and Plancherel in $y$, to give 
\begin{align*}
|\langle T_{0+i\mu}f , g \rangle_{L^{2}}| &= |\langle \tilde{S}_{\eps}(t)\big((|4 \pi t|^{-1}\widehat{P}_{\eps})^{\frac{-i(d-k)\mu}{2}}\tilde{f}\big), \tilde{g} \rangle_{L^{2}}|\\
&= \|e^{\frac{-i(d-k)\mu}{2}\ln(|4 \pi t|^{-1}\widehat{P}_{\eps})}\tilde{f}\|_{L_{x}^{2}L_{\eta}^{2}}\|g\|_{L^{2}} = \|f\|_{L_{x}^{2}L_{y}^{2}}\|g\|_{L_{x}^{2}L_{y}^{2}}.
\end{align*}
The estimate for $z=1+ i\mu$ follows similarly, but now using \eqref{eq:102}, so that 
\begin{align*}
|\langle T_{1+i\mu}f , g \rangle_{L^{2}}| &= |4 \pi t|^{\frac{(d-k)}{2}}|\langle \tilde{S}_{\eps}(t)\big((|4 \pi t|^{-1}\widehat{P}_{\eps})^{-i\frac{(d-k)\mu}{2}}\widehat{P}_{\eps}^{-\frac{(d-k)}{2}}\tilde{f}\big), \tilde{g} \rangle_{L^{2}}|\\
&\le \|\widehat{P}_{\eps}^{\frac{(d-k)}{2}}e^{\frac{-i(d-k)\mu}{2}\ln(|4 \pi t|^{-1}\widehat{P}_{\eps})}\widehat{P}_{\eps}^{-\frac{(d-k)}{2}}\tilde{f}\|_{L_{x}^{1}L_{\eta}^{2}}\|\tilde{g}\|_{L_{x}^{1}L_{\eta}^{2}} \\
&\le \|f\|_{L_{x}^{1}L_{y}^{2}}\|g\|_{L_{x}^{1}L_{y}^{2}}.
\end{align*}
Noting that the constants produce no growth in $z \in \C$, then the quoted version of Stein interpolation in \cite{StWe} implies for $0 \le \theta =1 - \frac{2}{r} \le 1$ and $r \in [2,\infty]$ the following estimate
$$
|4 \pi t|^{\delta(r)}\|P_{\eps}^{-\delta(r)}f\|_{L^{r}_{x}L^{2}_{y}} = \|T_{\theta}f\|_{L^{r}_{x}L^{2}_{y}} \le \|f\|_{L^{r'}_{x}L^{2}_{y}},
$$
which by replacing $f$ by $P_{\eps}^{\delta(r)}f$ and dividing the above inequality by $|4 \pi t|^{\delta(r)}$ gives \eqref{dualest}. 
Again, we note that there are no $\eps$-dependent constants arising from this interpolation step. 
Moreover, since the proof of this theorem exploits a density argument using simple functions, the result directly applies also to the mixed spaces $L^r_xL^2_y$ under consideration. 
\end{proof}

\begin{remark}
Note that, as $\eps \to 0$, the estimate \eqref{DisE} converges to
\[
\big \| S_0(t) f \big \|_{L^{r}_{x}L^{2}_{y}} \le  |4 \pi t|^{-(d-k)\left(\frac12-\frac{1}{r}\right)} \big \| f \big \|_{L^{r'}_{x}L^{2}_{y}},
\]
which is similar to the usual dispersion estimate for the Schr\"odinger group in dimension $d-k\in \N$ and again reflects the fact that we don't obtain dispersion in the $y$-coordinates 
when $\eps >0$. Deriving estimate \eqref{DisE1} from \eqref{DisE} has the advantage that we can use standard Sobolev spaces $H^s$, 
independent of $\varepsilon$, to measure the regularity in $y$ (instead of employing the operator $P_\eps$).
The price to pay is that \eqref{DisE1} no longer converges to the classical dispersion estimate in the limit $\eps \to 0$ (except in the case $r=2$ for which $\delta(r) = 0$). 
But since in this work
we are not concerned with the limit $\eps \to 0$, we shall ignore this issue in the following and base our Strichartz estimates on \eqref{DisE1}. 
\end{remark}

\subsection{Strichartz estimates}
Exploiting the dispersion estimate \eqref{DisE1}, we shall now prove space-time Strichartz estimates associated to $S_\eps(t)$. These estimates also follow from abstract arguments as in \cite{BCD, GV3, KT}. 
For the sake of concreteness and due to our somewhat unusual function spaces, we shall give their proof in the nonendpoint case. 

\begin{remark} The case of endpoint Strichartz estimates, i.e., $(q, r)=\big(2, \frac{2(d-k)}{(d-k-2)_{+}}\big)$ for $d-k\ge 2$, in principle could also be dealt 
with as in \cite{KT}, but since we never make use of it in our analysis, 
we shall not pursue this issue any further. 
\end{remark}

\begin{proposition}[Strichartz estimates]\label{prop:STR}
\ Let $S_{\eps}(t) = e^{it P_\eps^{-1} \Delta}$ and $(q, r), (\gamma, \rho)$ be two arbitrary admissible Strichartz pairs with $0<\delta(r), \delta(\rho)<1$. 
Then for any time interval $I$, there exist constants $C_1, C_2>0$, independent of $\eps$ and $I$, such that
\begin{equation}\label{S1}
\| S_\eps(\cdot)f \|_{L^{q}_{t}L^{r}_{x}H^{-\delta(r)}_{y}} \le C_1\| f \|_{L^{2}}, 
\end{equation}
as well as
\begin{equation}\label{S2}
\left\|\int_{0}^{t}S_\eps(\cdot-s)F(s)\;ds \right\|_{L^{q}_{t}L^{r}_{x}H^{-\delta(r)}_{y}} \le C_2\| F \|_{ L^{\gamma^{\prime}}_{t}L^{\rho^{\prime}}_{x}H^{\delta(\rho)}_{y} }. 
\end{equation}
\end{proposition}
\begin{proof}
We start by first noticing that \eqref{S1} is equivalent to saying that the map $f\mapsto S_\eps(t)f$ is bounded as an operator $L^2 \to L^q_tL^r_xH^{-\delta(r)}_y$. 
Let us define the operator $T_\eps:L^{q'}_tL^{r'}_xH_y^{\delta(r)}\to L^2$ by
\begin{equation*}
T_\eps F=\int_\R S_\eps(-s)F(s)\,ds
\end{equation*}
and note that its formal adjoint $T_\eps^{\ast}$ is the map $ f \mapsto S_\eps(t)f$. Next, we shall show that
\begin{equation*}
T_\eps^\ast T_\eps F(t)=\int_\R S_\eps(t-s)F(s)\,ds
\end{equation*}
is bounded as an operator $L^{q'}_tL^{r'}_xH^{\delta(r)}_y \to L^q_tL^r_xH^{-\delta(r)}_y$. By the generalized Minkowski's inequality we have 
\begin{equation*}
\left\|\int_\R S_\eps(\cdot-s)F(s)\,ds \right\|_{L^q_tL^r_xH^{-\delta(r)}_y}\le \left\|\int_\R\|S_\eps(\cdot-s)F(s)\|_{L^r_xH^{-\delta(r)}_y}\,ds\right\|_{L^q_t},
\end{equation*}
and applying the dispersion estimate \eqref{DisE1}, it follows that 
\begin{equation*}
\|S_\eps(t-s)F(s)\|_{L^r_xH^{-\delta(r)}_y}\le |4\pi(t-s)|^{-\delta(r)}\|F(s)\|_{L^{r'}_xH^{\delta(r)}_y}.
\end{equation*}
Hence recalling that $\delta(r)=\frac2q < 1$, we see it is then possible to apply the Hardy--Littlewood--Sobolev inequality in order to obtain 
\begin{equation*}\begin{aligned}
\left\|\int_\R S_\eps(\cdot-s)F(s)\,ds\right\|_{L^q_tL^r_xH^{-\delta(r)}_y}& \le\left\|\int_\R |4\pi(\cdot-s)|^{-\delta(r)}\|F(s)\|_{L^{r'}_xH^{\delta(r)}_y}\,ds\right\|_{L_t^q}\\
&\le C\|F\|_{L^{q'}_tL^{r'}_xH^{\delta(r)}_y}.
\end{aligned}\end{equation*}
We thus have proven that the operator $T_\eps^\ast T_\eps:L^{q'}_tL^{r'}_xH^{\delta(r)}_y\to L^q_tL^r_xH^{-\delta(r)}_y$ is bounded. 
A standard functional analysis result for operators on Banach spaces (see, e.g., \cite{BCD}) states that
\begin{equation*}
\|T_\eps\|^2_{\mathcal L(L^{q'}_tL^{r'}_xH_y^{\delta(r)};L^2)}=\|T_\eps^\ast\|^2_{\mathcal L(L^2:L^q_tL^r_xH^{-\delta(r)}_y)}
=\|T_\eps^\ast T_\eps\|_{\mathcal L(L^{q'}_tL^{r'}_xH^{\delta(r)}_y;L^q_tL^r_xH^{-\delta(r)}_y)}.
\end{equation*}
This consequently implies that both $$T_\eps:L^{q'}_tL^{r'}_xH^{\delta(r)}_y\to L^2 \quad \text{and}\quad
T_\eps^\ast:L^2\to L^q_tL^r_xH^{-\delta(r)}_y$$ are bounded with norms independent of $\eps$. In particular, \eqref{S1} is proved. Furthermore, we note that this holds for any nonendpoint admissible pair $(q, r)$. 

Now, choose any arbitrary (nonendpoint) admissible pairs $(\gamma, \rho)$ and $(q, r)$ 
such that $$T_\eps:L^{\gamma'}_tL^{\rho'}_{x}H^{\delta(\rho)}_{y}\to L^{2} \quad \text{and} \quad T_\eps^{\ast}:L^{2} \to L^{q}_{t}L^{r}_{x}H^{-\delta(r)}_{y}.$$ 
By combining the estimates for the operators $T_\eps$, $T_\eps^\ast$, we then infer that 
$$T_\eps^\ast T_\eps:L^{\gamma'}_{t}L^{\rho'}_{x}H^{\delta(\rho)}_{y} \to L^{q}_{t}L^{r}_{x}H^{-\delta(r)}_y $$
is bounded, i.e.,
\begin{equation*}
\left\|\int_\R S_\eps(\cdot-s)F(s)\,ds\right\|_{L^{q}_{t}L^{r}_{x}H^{-\delta(r)}_{y}}\le C\|F\|_{L^{\gamma'}_{t}L^{\rho'}_{x}H_{y}^{\delta(\rho)}},
\end{equation*}
for any arbitrary $(q, r), (\gamma, \rho)$. We can then invoke Theorem $1.2$ from the paper \cite{CK} by Christ and Kiselev to conclude the retarded estimate
\begin{equation*}
\left\|\int_{s<t}S_\eps(\cdot-s)F(s)\,ds\right\|_{L^{q}_{t}L^{r}_{x}H^{-\delta(r)}_{y}}\le C\|F\|_{L^{\gamma'}_{t}L^{\rho'}_{x}H^{\delta(\rho)}_{y}}.
\end{equation*}
In summary, this proves the desired result. 
\end{proof}


\section{The Cauchy problem for partial off-axis variation in the subcritical case}\label{sec:CP}


In this section we shall give the proof of Theorem \ref{thm:sub} by proving a global $L^2$-based well-posedness result for \eqref{PNLS} with subcritical nonlinearities. 
In a second step we shall establish the additional $H^1$-regularity of the solution.

\subsection{Well-posedness in terms of $v$}\label{subsc:WPL2}

We rewrite \eqref{PNLS} using Duhamel's formulation, i.e.,
\begin{equation}\label{intPNLS}
v(t) = S_\eps(t) v_{0} + i\int_{0}^{t} S_{\eps}(t-s)P_{\eps}^{-1/2}(|P_{\eps}^{-1/2}v|^{2\sigma}P_{\eps}^{-1/2}v)(s) \;ds=: \Phi(v)(t).
\end{equation}
For the sake of brevity, we shall also write
\[
\Phi(v)(t) = S_\eps(t)v_{0} + \mathcal{N}(v)(t)
\]
and denote
\begin{equation}\label{nlterm}
\mathcal{N}(v)(t) := i\int_{0}^{t} S_\eps(t-s)P_{\eps}^{-1/2}g(P_{\eps}^{-1/2}v(s)) \;ds,
\end{equation}
where $g(z) = |z|^{2\sigma}z$ with $\sigma >0$.

Of course, the basic idea is to prove that $v\mapsto \Phi(v)$ is a contraction mapping in a suitable Banach space. To this end, the following lemma is key.

\begin{lemma}\label{lem:Nlest}
Let $d-k>0$. Fix $T>0$ and choose the admissible pair $$(\gamma,\rho)= \left(\frac{4(\sigma+1)}{(d-k)\sigma},2(\sigma+1)\right).$$ 
Then, in the space-time slab $\R^{d}\times [0,T] $ the inequality 
\begin{align*}
& \, \|\mathcal{N}(v)-\mathcal{N}(v')\|_{L^{\gamma}_{t}L^{\rho}_{x}H^{-\delta(\rho)}_{y}} \\
& \, \lesssim \eps^{-2(\sigma+1)}T^{1-\frac{(d-k)\sigma}{2}}\left(\|v\|_{L^{\gamma}_{t}L^{\rho}_{x}H^{-\delta(\rho)}_{y}}^{2\sigma}+\|v'\|_{L^{\gamma}_{t}L^{\rho}_{x}H^{-\delta(\rho)}_{y}}^{2\sigma}\right)
\|v-v'\|_{L^{\gamma}_{t}L^{\rho}_{x}H^{-\delta(\rho)}_{y}} ,
\end{align*}
holds, provided $0< \sigma \le \frac{2}{(d-2)_{+}} $.
\end{lemma}

The case $k=0$ is classical and thus we will only give the proof for $d>k>0$.
\begin{proof} 
We first note that for our pair $(\gamma,\rho)$ to be non-endpoint admissible for $d-k\ge 2$, we require that $\gamma >2$, which in turn is equivalent to 
$\sigma <\frac{2}{(d-k-2)_{+}}$. However, this condition will always be fulfilled since
\[
\sigma \le \frac{2}{(d-2)_{+}}  < \frac{2}{(d-k-2)_+}.
\]
As a consequence, we also have that $\delta(\rho) = \frac{(d-k)\sigma}{2(\sigma +1)}<1$. 

Now, as a first step we apply the Strichartz estimate \eqref{S2} and note that
\begin{align*}
 \|\mathcal{N}(v)&-\mathcal{N}(v')\|_{L^{\gamma}_{t}L^{\rho}_{x}H^{-\delta(\rho)}_{y}} \\
 &\le C_{2}\| P_{\eps}^{-1/2}(g(P_{\eps}^{-1/2}v)-g(P_{\eps}^{-1/2}v'))\|_{L^{\gamma'}_{t}L^{\rho'}_{x}H^{\delta(\rho)}_{y}} \\
  &\le \eps^{-1}C_{2} \| g(P_{\eps}^{-1/2}v)-g(P_{\eps}^{-1/2}v')\|_{L^{\gamma'}_{t}L^{\rho'}_{x}H^{-(1-\delta(\rho))}_{y}},
\end{align*}
where we have also used the scaling \eqref{act2} to obtain the factor $\eps^{-1}$. Next, by a Sobolev embedding we have that $ H^{s}(\R^{k})  \hookrightarrow L^{\rho}(\R^{k})$, where 
$$
s = k\left(\frac{1}{2} - \frac{1}{2(\sigma+1)}\right) = \frac{k\sigma}{2(\sigma+1)} \in \Big(0,\frac{k}{2} \Big).
$$
In turn, this also implies the dual embedding $L^{\rho'}(\R^{k}) \hookrightarrow H^{-s}(\R^{k})$.
Now, if we impose that 
$$
1 \ge s + \delta(\rho)= \frac{d\sigma}{2(\sigma+1)},
$$
which is so whenever $\sigma \le \frac{2}{(d-2)_{+}}$, then $H^{-s}(\R^{k})\hookrightarrow H^{-(1-\delta(\rho))}(\R^{k})$.
Together these allow us to estimate
 \begin{align*}
\| g(P_{\eps}^{-1/2}v)&-g(P_{\eps}^{-1/2}v')\|_{H^{-(1-\delta(\rho))}_{y}} \le \| g(P_{\eps}^{-1/2}v)-g(P_{\eps}^{-1/2}v')\|_{H^{-s}_{y}}\\
&\le C_{\sigma}\| (|P_{\eps}^{-1/2}v|^{2\sigma} + |P_{\eps}^{-1/2}v'|^{2\sigma})P_{\eps}^{-1/2}(v - v')  \|_{L^{\rho'}_{y}} =(\ast),
\end{align*}
where we have also used that for all $z,w \in \C$, 
$$|g(z) - g(w)| \le C_{\sigma}(|z|^{2\sigma}  + |w|^{2\sigma} )|z-w|.
$$
Now, recall that $\rho =  2(\sigma+1)$ and hence $\frac{1}{\rho'}= \frac{2\sigma}{\rho}+\frac{1}{\rho}$. Thus, by first applying H\"older's inequality and using \eqref{act2}, we obtain
\begin{align*}
(\ast) &\lesssim (\|P_{\eps}^{-1/2}v\|_{L^{\rho}_{y}}^{2\sigma} + \|P_{\eps}^{-1/2}v'\|_{L^{\rho}_{y}}^{2\sigma})\|P_{\eps}^{-1/2}(v - v')  \|_{L^{\rho}_{y}}\\
 &\lesssim \eps^{-(2\sigma+1)}(\|v\|_{H^{-(1-s)}_{y}}^{2\sigma} + \|v'\|_{H^{-(1-s)}_{y}}^{2\sigma})\| v - v'  \|_{H^{-(1-s)}_{y}}\\
&\lesssim \eps^{-(2\sigma+1)}(\|v\|_{H^{-\delta(\rho)}_{y}}^{2\sigma} + \|v'\|_{H^{-\delta(\rho)}_{y}}^{2\sigma})\| v - v'  \|_{H^{-\delta(\rho)}_{y}},
\end{align*}
where the last inequality follows from $H^{-\delta(\rho)}(\R^{k}) \hookrightarrow H^{-(1-s)}(\R^{k})$, by the same arguments as before. 
Employing H\"older's inequality once more in $x$, we consequently infer
\begin{align*}
\| g(P_{\eps}^{-1/2}v)&-g(P_{\eps}^{-1/2}v')\|_{L^{\rho'}_{x}H^{-(1-\delta(\rho))}_{y}}\\
&\lesssim \eps^{-(2\sigma+1)}(\|v\|_{L^{\rho}_{x}H^{-\delta(\rho)}_{y}}^{2\sigma} + \|v'\|_{L^{\rho}_{x}H^{-\delta(\rho)}_{y}}^{2\sigma})\| v - v'  \|_{L^{\rho}_{x}H^{-\delta(\rho)}_{y}}.
\end{align*}
From here, we compute that
$$
\frac{1}{\gamma'} = 1-\frac{(d-k)\sigma}{2} +\frac{2\sigma}{\gamma} + \frac{1}{\gamma} \, .
$$
Thus, taking the $L^{\gamma'}$ norm in $t$ and applying H\"older's inequality yields the result of the lemma.
\end{proof}

Using Lemma \ref{lem:Nlest}, we are now able to prove global well-posedness for \eqref{PNLS} in the subcritical case. 
In doing so, we will require a positive exponent $$\alpha \equiv {1-\frac{(d-k)\sigma}{2}}$$ of $T$ in the estimate obtained in Lemma \ref{lem:Nlest}, i.e., we require $\sigma < \frac{2}{d-k}.$
Since Lemma \ref{lem:Nlest} holds for $\sigma \le \frac{2}{(d-2)_{+}}$, we need to distinguish the cases $k \le 2$ and $k>2$ in the following.

One notices immediately that for $k\le 2$, we have that $\frac{2}{d-k} \le \frac{2}{(d-2)_{+}}$, which in turn 
implies that, in this case, we require the stronger assumption $\sigma < \frac{2}{d-k}$ to ensure $\alpha >0$. However, for $k>2$ (and thus $d>3$), it holds that
\[
\frac{2}{d-2} < \frac{2}{d-k}<\frac{2}{(d-k-2)_+},
\] 
and hence no new restriction arises. We also note that for $k>2$, the exponent of $T^\alpha$ is positive and is $L^{2}$-subcritical 
in the sense that when $\sigma = \frac{2}{d-2}$ then 
$$
\alpha = 1 - \frac{(d-k)\sigma}{2} = \frac{k-2}{d-2}>0.
$$

With this in mind, we can now prove the following result.

\begin{proposition}\label{prop:GWPL2}
Let $d> k\ge 0$ and
\begin{itemize}
\item either $k\le 2$ and $0 \le \sigma < \frac{2}{d-k}$ 
\item or $k>2$ and $0 \le \sigma \le \frac{2}{d-2}$.
\end{itemize}
Then for any $v_{0} \in L^{2}(\R^{d})$, there exists a unique global solution to \eqref{PNLS}
\[ 
v \in C(\R_t,L^{2}(\R^{d}))\cap L^{q}_{\rm loc}(\R_t;L^{r}(\R_x^{d-k}; H^{-\delta(r)}(\R_{y}^{k})))
\]
for any (nonendpoint) admissible pair $(q, r)$. Moreover, $v$ depends continuously on the initial data and satisfies 
\[
\|v(t,\cdot)\|_{L^{2}} = \|v_{0}\|_{L^{2}} \quad \forall \, t\in\R.
\]
\end{proposition}
By identifying $v=P_\eps^{1/2} u$, this directly yields a global-in-time solution $u\in C(\R;L^{2}(\R_{x}^{d-k}; H^{1}(\R_{y}^k)))$ to \eqref{NLS} 
and thus proves Theorem \ref{thm:sub}. Note that here continuous dependence on the initial data precisely 
means that for $T>0$ the map $v_0\mapsto v{\vert}_{[-T,T]}$ is continuous 
as a map
\[
L^2(\R^d) \to C([-T,T], L^2(\R^d))\cap L^q([-T,T]; L^r(\R^{d-k}_x; H^{-\delta(r)} (\R^k_y))).
\]

\begin{proof}
We shall prove Proposition \ref{prop:GWPL2} in several steps.

Step 1 (Existence):
Fix the admissible pair $(\gamma,\rho)= \Big(\frac{4(\sigma+1)}{(d-k)\sigma},2(\sigma+1)\Big)$. Let $M,T>0$ to be determined later and denote $I=[0,T]$, and set 
\begin{align*}
X_{T,M} = \{ v \in & L_{t}^{\infty}L^{2}(I) \cap L_{t}^{q}L_{x}^{r}H^{-\delta(r)}_{y}(I): \|v \|_{L_{t}^{\infty}L^{2}} + \|v \|_{L^{\gamma}_{t}L^{\rho}_{x}H^{-\delta(\rho)}_{y}} \le M \}.
\end{align*}
We note that $X_{T,M}$ is a complete metric space equipped with the distance 
$$
d(v,w) = \|v-w\|_{L_{t}^{\infty}L^{2}_{x,y}} + \|v-w\|_{L^{\gamma}_{t}L^{\rho}_{x}H^{-\delta(\rho)}_{y}}.
$$
Let $v \in X_{T,M}$. Then the Strichartz estimates obtained in Proposition \ref{prop:STR} together with Lemma \ref{lem:Nlest} imply that 
\begin{align*}
\| \Phi(v) \|_{L^{\gamma}_{t}L^{\rho}_{x}H^{-\delta(\rho)}_{y} } &\le\| S_{\eps}(\cdot)v_{0} \|_{L^{\gamma}_{t}L^{\rho}_{x}H^{-\delta(\rho)}_{y}} + \|\mathcal{N}(v)\|_{L^{\gamma}_{t}L^{\rho}_{x}H^{-\delta(\rho)}_{y}} \\
&\le C_{\sigma,\eps}\Big(\| v_{0} \|_{L^{2}} +  T^{1-\frac{(d-k)\sigma}{2}}M^{2\sigma+1} \Big),
\end{align*}
as well as
\begin{align*}
\| \Phi(v) \|_{L^{\infty}_{t}L^{2} } &\le \| v_{0} \|_{L^{2}} + C_{2}\|P_{\eps}^{-1/2}g(P_{\eps}^{-1/2}v)\|_{L^{\gamma'}_{t}L^{\rho'}_{x}H^{\delta(\rho)}_{y}} \\
&\le C_{\sigma,\eps}\Big(\| v_{0} \|_{L^{2}} +  T^{1-\frac{(d-k)\sigma}{2}}M^{2\sigma+1} \Big).
\end{align*}
Together, these yield
$$
\|\Phi(v) \|_{L_{t}^{\infty}L^{2}} + \|\Phi(v) \|_{L^{\gamma}_{t}L^{\rho}_{x}H^{-\delta(\rho)}_{y}}  \le 2C_{\sigma,\eps}\Big( \| v_{0} \|_{L^{2}} +  T^{1-\frac{(d-k)\sigma}{2}}M^{2\sigma+1}\Big). 
$$
We now choose $M$ such that 
\begin{equation*}
3M = 8C_{\sigma,\eps}\|v_{0}\|_{L^{2}}
\end{equation*}
and choose $T>0$ such that
\begin{equation}\label{eq:T}
2C_{\sigma,\eps}T^{1-\frac{(d-k)\sigma}{2}}M^{2\sigma+1} \le \frac{M}{4}.
\end{equation}
Then it follows that $\Phi(v) \in X_{T,M}$ for all $v \in X_{T,M}$ so that $\Phi(X_{T,M}) \subset X_{T,M}$. 
Now, let $v,w \in X_{T,M} $. Then by Lemma \eqref{lem:Nlest} and using \eqref{eq:T} we have
\begin{align}\label{StrNL}
\| \mathcal{N}(v) - \mathcal{N}(w) \|_{L_{t}^{\gamma}L_{x}^{\rho}H_{y}^{-\delta(\rho)} } &\le 2C_{\sigma,\eps}M^{2\sigma}T^{1-\frac{(d-k)\sigma}{2}}\|v-w\|_{L^\gamma_tL^\rho_xH^{-\delta(\rho)}_y}  \\
& \le \frac{1}{4}\|v-w\|_{L^\gamma_tL^\rho_xH^{-\delta(\rho)}_y}, \nonumber
\end{align}
which together with the same estimate for the $L_{t}^{\infty}H^{1}$-norm gives 
$$ 
d(\Phi(v),\Phi(w)) \le \frac{1}{2}d(v,w), \quad \forall v,w \in X_{T,M}.
$$
Thus $\Phi$ is a contraction map on $X_{T,M}$ and Banach's fixed point theorem yields the existence of a unique fixed point $v \in X_{T,M}$. Furthermore, since the solution $v$ satisfies the integral equation \eqref{intPNLS}, we infer continuity in time, i.e., $v\in C(I;L^{2}(\R^{d}))$.

Moreover, if $v \in X_{T,M}$, then $v \in L_{t}^{q}L_{x}^{r}H^{-\delta(r)}_{y}(I)$ for any admissible pair $(q,r)$, since by our Strichartz estimates
$$
\|v\|_{L^{q}_{t}L^{r}_{x}H^{-\delta(r)}_{y}   } \equiv \|\Phi(v)\|_{L^{q}_{t}L^{r}_{x}H^{-\delta(r)}_{y}   } \le C_{1}\|v_0\|_{L^2} + C_{2}\|P_{\eps}^{-1/2}g(P_{\eps}^{-1/2}v)\|_{L^{\gamma'}_{t}L^{\rho'}_{x}H^{\delta(\rho)}_{y}   },
$$
which is estimated as in the proof of Lemma \ref{lem:Nlest}.

Step 2 (Uniqueness):  
Let $I =[0,T]$ and $v,w \in C(I;L^{2})\cap L_{t}^{q}L_{x}^{r}H^{-\delta(r)}_{y}(I)$ be two solutions to 
\eqref{intPNLS} with $\varphi = v_{0} =w_{0}$. Then as in Step 1, we have  $v, w\in X_{T, M}$ with 
$3M=8C_{\sigma, \eps}\|\varphi\|_{L^2}$ and $T$ given by \eqref{eq:T}. Since the difference of $v$ and $w$ is given by
\begin{equation*}
(v-w)(t)=\mathcal N(v)(t)-\mathcal N(w)(t),
\end{equation*}
then we can apply \eqref{StrNL} from Step 1 on the interval $I$ to obtain
\begin{equation*}
\|v-w\|_{L^\gamma_tL^\rho_xH^{-\delta(\rho)}_y(I)}\leq \frac{1}{4}\|v-w\|_{L^\gamma_tL^\rho_xH^{-\delta(\rho)}_y(I)}.
\end{equation*}
From this we conclude (local) uniqueness
\begin{equation*}
\|v-w\|_{L^\gamma_tL^\rho_xH^{-\delta(\rho)}_y(I)}=0,
\end{equation*}
i.e., $v=w$ on $I=[0,T]$. 

In addition, the solution depends continuously on the initial data, as can be seen by taking two solutions $v,\tilde{v}$ on a common time interval $I_{c}= \min \{ I, \tilde I \}$. Then by what was done above, we have that $v,\tilde{v} \in X_{T,M}$ with $3M=8 \max\{\|v_{0}\|_{L^{2}},\|\tilde{v}_{0}\|_{L^{2}}  \}$ and $T = |I_{c}|$ satisfying \eqref{eq:T} so that
$$
d(v,\tilde{v}) \le \|v_{0}-\tilde{v}_{0}\|_{L^{2}} + \frac{1}{2} d(v,\tilde{v}),
$$
which proves the continuous dependence on the initial data, after extending the argument to the interval $I_{c}$.

Step 3 (Global existence):
In order to show that the solution obtained in Step 1 indeed exists for all times $t\in \R$, let
$$
T_{\rm \max} = \sup\{T>0 :  \text{there exists a solution $v(t, \cdot)$ on $[0,T)$} \}.
$$
We claim that 
$$
\text{if} \quad T_{\rm{max}}<+\infty, \quad \text{then} \quad \lim_{t \to T_{\rm{max}} } \| v (t) \|_{L^{2}} = +\infty. 
$$
Suppose, by contradiction, that $T_{\rm \max}<\infty$ and that there exists a sequence $ t_{j} \to T_{\rm \max}$ such that $\|v(t_{j})\|_{L^{2}} \le M$. 
Now choose some integer $J$ such that $t_{J}$ is close to $T_{\rm \max}$ where by assumption $\|v(t_{J})\|_{L^{2}} \le M$.  But by Step 1, using the initial data $v(t_{J})$ we can 
extend our solution to the interval $[t_{J}, t_{J}  +T]$ where we now choose $t_{J}$ such that
$$
t_{J}  +T > T_{\rm max}.
$$
This gives a contradiction to the definition of $T_{\rm max}$.

Next, we shall prove that the $L^2$-norm of $v$ is conserved along the time-evolution. To this end, we adapt an elegant argument given in \cite{Oz}, which has the 
advantage that it does not require an approximation procedure using a sequence of sufficiently smooth solutions (as is classically done, see, e.g., \cite{Cz}). 
First note that by Step 1 we have $v \in C([0,T]; L^{2}(\R^{d}))$ for any $T<T_{\rm max}$. We then 
rewrite Duhamel's formula \eqref{intPNLS}, using the continuity of the semigroup $S_\eps$ to propagate backward in time
\begin{equation}\label{eq:N}
S_\eps(-t)v(t) =  v_{0} + S_\eps(-t)\mathcal{N}(v)(t).
\end{equation}
The fact that $S_{\eps}(\cdot)$ is unitary in $L^{2}$ implies $\|v(t)\|_{L^{2}} = \|S_\eps(-t)v(t)\|_{L^{2}}$. The latter can be expressed using 
the above identity to obtain
\begin{align*}
\|v(t)\|^{2}_{L^{2}} &\, = \|v_{0}\|^{2}_{L^{2}} +2\re \, \big \langle S_\eps(-t)\mathcal{N}(v)(t),v_{0} \big\rangle_{L^{2}} + \| S_\eps(-t)\mathcal{N}(v)(t) \|^2_{L^{2}}\\
& \, =: \|v_{0}\|^{2}_{L^{2}}  + \mathcal{I}_{1} + \mathcal{I}_{2}. 
\end{align*}
We want to show that $\mathcal{I}_{1} + \mathcal{I}_{2}=0$.
In view of \eqref{nlterm} we can rewrite
\begin{align*}
\mathcal{I}_{1} &= -2\im \, \big\langle \int_{0}^{t} S_{\eps}(-s)P_{\eps}^{-1/2}g(P_{\eps}^{-1/2}v)(s)\;ds , v_0  \big\rangle_{L^2} \\
&=  -2\im  \int_{0}^{t} \big\langle  P_{\eps}^{-1/2}g(P_{\eps}^{-1/2}v)(s) , S_{\eps}(s)v_0  \big\rangle_{L^2} \;ds.
\end{align*}
By duality in $y$ and H\"older's inequality in both $x$ and $t$ we find that this quantity is indeed finite
\begin{align*}
|\mathcal{I}_{1}| \le 2\|  P_{\eps}^{-1/2}g(P_{\eps}^{-1/2}v)\|_{L^{\gamma'}_{t}L_{x}^{\rho'}H_{y}^{\delta(\rho)}}\| S_{\eps}(s)v_0 \|_{L^{\gamma}_{t}L_{x}^{\rho}H_{y}^{-\delta(\rho)}} <\infty.
\end{align*}
Denoting for simplicity $G_{\eps}(\cdot) = P_{\eps}^{-1/2}g(P_{\eps}^{-1/2}v)(\cdot)$, we perform the following computation:
\begin{align*}
\mathcal{I}_{2}& \equiv \big\langle \int_{0}^{t} S_{\eps}(-s)G_{\eps}(s) \;ds, \int_{0}^{t} S_{\eps}(-s')G_{\eps}(s') \;ds' \big\rangle_{L^{2}} \\
&=\int_{0}^{t} \big\langle  S_{\eps}(-s)G_{\eps}(s) , \Big(\int_{0}^{s} + \int_{s}^{t} \Big) S_{\eps}(-s')G_{\eps}(s') \;ds' \big\rangle_{L^{2}}\;ds \\
&= \int_{0}^{t} \big\langle  G_{\eps}(s) , \int_{0}^{s} S_{\eps}(s-s')G_{\eps}(s') \;ds' \big\rangle_{L^{2}}\;ds  \\
&\quad \quad \quad \quad \quad + \int_{0}^{t}\int_{s}^{t}  \big\langle  S_{\eps}(s'-s)G_{\eps}(s) ,   G_{\eps}(s')  \big\rangle_{L^{2}}\;ds'\;ds \\
&=\int_{0}^{t} \big\langle  G_{\eps}(s) , -i\mathcal{N}(v)(s) \big\rangle_{L^{2}}\;ds + \int_{0}^{t}  \big\langle  \int_{0}^{s'} S_{\eps}(s'-s)G_{\eps}(s)\;ds,   G_{\eps}(s')  \big\rangle_{L^{2}}\;ds'\\
&= 2 \re \int_{0}^{t} \big\langle  G_{\eps}(s) , -i\mathcal{N}(v)(s) \big\rangle_{L^{2}}\;ds.
\end{align*}
Using the integral formulation \eqref{eq:N}, we can express $-i\mathcal{N}(v)(s)$ and write
\begin{equation}\label{yeta}
\mathcal{I}_{2} = 2 \re \, \Big( \int_{0}^{t} \big\langle  G_{\eps}(s) , iS_{\eps}(s)v_0 \big\rangle_{L^{2}}\;ds + \int_{0}^{t} {\big\langle  G_{\eps}(s) , -iv(s) \big\rangle_{L^{2}}}\;ds \Big).
\end{equation}
Here we note that the particular form of our nonlinearity implies
$$
\re\,\big\langle G_{\eps}(\cdot),  -iv(\cdot) \big\rangle_{L^{2}} = \im \, \big\langle g(P_{\eps}^{-1/2}v)(\cdot),  P_{\eps}^{-1/2}v(\cdot) \big\rangle_{L^{2}}
= \im \, \|P_{\eps}^{-1/2}v(\cdot)\|^{2\sigma +2}_{L^{2\sigma +2}} =0,  
$$
and thus the second term on the right-hand side of \eqref{yeta} simply vanishes. In summary, we find
\begin{align*}
\mathcal{I}_{2} = 2 \re \, \int_{0}^{t} \big\langle  G_{\eps}(s) , iS_{\eps}(s)v_0 \big\rangle_{L^{2}}\;ds 
= 2 \im \, \int_{0}^{t} \big\langle  S_{\eps}(-s)G_{\eps}(s) \;ds , v_0 \big\rangle_{L^{2}} \equiv -\mathcal{I}_{1},
\end{align*}
which proves that 
$$
\|v(t)\|_{L^{2}} = \|v_{0}\|_{L^{2}} \quad \forall t \in [0,T].
$$
This conservation law allows us to reapply Step 1 as many times as we wish, thereby preserving the length of the maximal interval in each iteration, 
and yielding $T_{\rm max} = +\infty$. Since the equation is time-reversible modulo complex conjugation, this yields a global solution for all $t\in \R$.
\end{proof}

\subsection{Higher order regularity}
In this subsection, we are going to prove that the global-in-time $L^2$-solution obtained in Proposition \ref{prop:GWPL2} 
enjoys persistence of regularity. Namely, if the initial datum $v_{0}\in H^1$, then the corresponding solution $v(t,\cdot)$ remains in $H^1$ for all times $t\in \R$. 
We will prove this property by exploiting the Strichartz estimates stated in Proposition \ref{prop:STR} and the global well-posedness result in $L^2$. 
Similar arguments can be used to obtain a solution $v(t, \cdot)\in H^s$,  $s\ge 1$, provided the nonlinearity is sufficiently smooth.

\begin{proposition}
Let $v \in C(\R_t,L^{2}(\R^{d}))\cap L^{q}_{\rm loc}(\R_t;L^{r}(\R_x^{d-k}; H^{-\delta(r)}(\R_{y}^{k})))$ be the solution obtained in Proposition \ref{prop:GWPL2} 
with initial data $v_0\in L^2(\R^d)$. If, in addition, 
$v_0 \in H^1(\R^d)$, then $v \in C(\R_t; H^1(\R^{d}))$.
\end{proposition}

\begin{proof}
Let us fix a $0<T<\infty$. We are going to show that
\begin{equation}\label{estgrad}
\|\nabla v\|_{L^\infty_tL^2([0, T])}\le K(T,\|\nabla v_0\|_{L^2}).
\end{equation}
Having in mind the conservation property of the $L^2$-norm of $v$, this estimate is sufficient to conclude the desired result. 

To obtain \eqref{estgrad}, we first recall from Proposition \ref{prop:GWPL2} that
\begin{equation*}
\|v\|_{L^\gamma_tL^\rho_xH^{-\delta(\rho)}_y([0, T])}\le C(T,\| v_0\|_{L^2})=:C_T,
\end{equation*}
where $(\gamma, \rho) = \Big(\frac{4(\sigma+1)}{(d-k)\sigma},2(\sigma+1)\Big)$ is the admissible pair used in Lemma \ref{lem:Nlest}. Let $\lambda>0$ 
be a small parameter to be chosen later on. We then divide $[0, T]$ into $N=N(\lambda, C_T)$ subintervals, i.e., $[0, T]=\cup_{j=1}^NI_j$, where $I_j=[t_{j-1}, t_j]$ and $0=t_0<t_1<\dotsc<t_N=T$, such that 
\begin{equation}\label{eq:small1}
\| v\|_{L^\gamma_tL^\rho_xH^{-\delta(\rho)}_y(I_j)}\le\lambda,\quad j=1, \dotsc, N.
\end{equation}
First we estimate the gradient of \eqref{nlterm} by a similar strategy as in Lemma \ref{lem:Nlest} with $v'=0$.   By applying the Strichartz estimate \eqref{S2} and the appropriate embeddings in $y$ gives
\begin{align*}
\|\nabla \mathcal{N}(v)\|_{L^{\gamma}_{t}L^{\rho}_{x}H^{-\delta(\rho)}_{y}(I_{j})} \le  \eps^{-1}C_{2}\| \nabla g(P_{\eps}^{1/2}v)  \|_{L_{t}^{\gamma'}L^{\rho'}_{x}L^{\rho'}_{y}}.
\end{align*}
Since the nonlinearity is smooth, this allows us to estimate in $y$ as follows:
\begin{align*}
\| \nabla g(P_{\eps}^{1/2}v)  \|_{L^{\rho'}_{y}} &\le (2\sigma +1)\|P_{\eps}^{-1/2}v  \|^{2\sigma}_{L_{y}^{\rho}}\|P_{\eps}^{-1/2} \nabla v   \|_{L_{y}^{\rho}}\\
&\lesssim \eps^{-(2\sigma +1)}\|v  \|^{2\sigma}_{H_{y}^{-\delta(\rho)}}\| \nabla v   \|_{H_{y}^{-\delta(\rho)}}.
\end{align*}
Combining this with a H\"older estimate in $x$ and $t$, similarly as in Lemma \ref{lem:Nlest} above, we obtain  
$$
\|\nabla \mathcal{N}(v)\|_{L^{\gamma}_{t}L^{\rho}_{x}H^{-\delta(\rho)}_{y}(I_{j})} \lesssim \eps^{-2(\sigma +1)} |I_{j}|^{1-\frac{(d-k)\sigma}{2}}\lambda^{2\sigma}
\|\nabla v\|_{L^{\gamma}_{t}L^{\rho}_{x}H^{-\delta(\rho)}_{y}(I_{j})}.
$$
Hence on each subinterval $I_{j}$ we have that
\begin{align*}
\|\nabla v\|_{L_{t}^{\infty}L^{2}(I_j)} &+ \|\nabla v\|_{L_{t}^{\gamma}L^{\rho}_{x}H_{y}^{-\delta(\rho)}(I_j)} \\
&\le C_{\eps}\Big( \|\nabla v_{j-1}\|_{L^{2}} + |I_j|^{1-\frac{(d-k)\sigma}{2}}\lambda^{2\sigma}
\|\nabla v\|_{L^{\gamma}_{t}L^{\rho}_{x}H^{-\delta(\rho)}_{y}(I_j)} \Big)
\end{align*}
for $j=1,\dotsc,N$ where we write $\nabla v_{j-1}$ to denote $\nabla v(t_{j-1})$.
Now choose $\lambda=\lambda({C_{\eps}}, T)$  such that
$$
{C_{\eps}}T^{1-\frac{(d-k)\sigma}{2}}\lambda^{2\sigma} <1.
$$
Since $|I_{j}| \le T$ we infer the estimate
$$
\|\nabla v\|_{L_{t}^{\infty}L^{2}(I_j)} + \|\nabla v\|_{L_{t}^{\gamma}L^{\rho}_{x}H_{y}^{-\delta(\rho)}(I_j)} \le K^\eps_j\| \nabla v_{j-1} \|_{L^{2}},
$$
for some constant $K^\eps_{j}$ which depends on $\eps$.
In particular, for $j=1,\dotsc,N$ we have 
\begin{equation*}
\|\nabla v_{j}\|_{L^2}\le K^\eps_j\|\nabla v_{j-1}\|_{L^2}.
\end{equation*}
Using this, we iterate the argument on each subinterval $I_j$, $j=1, \dotsc, N$, to obtain the desired estimate \eqref{estgrad}.
\end{proof}

\begin{remark} Notice that we cannot obtain uniform-in-time bounds on the $H^{1}$-norm of $v$ by invoking the energy \eqref{CLE}. 
Indeed the energy functional, written in terms of $v$, reads
\begin{equation*}
E(t)=\frac{1}{2}\|P^{-1/2}_\eps\nabla v \|^{2}_{L^2} -\frac{1}{2(\sigma +1)}\| P^{-1/2}_\eps v \|^{2\sigma +2}_{L^{2\sigma +2}},
\end{equation*}
which cannot provide a uniform bound on the full gradient of $v$.
\end{remark}

The proposition above yields a solution $u$ to \eqref{NLS} 
such that $v(t, \cdot)=P^{1/2}_\eps u(t, \cdot)\in H^1(\R^{d})$ globally in time. In particular, since $$\|u(t, \cdot)\|_{H^1}\le\|P^{1/2}_\eps u(t, \cdot)\|_{H^1},$$ 
we infer $u(t,\cdot) \in H^1(\R^d)$ for all $t\in \R$, 
provided $P^{1/2}_\eps u_0\in H^1$. This shows that for a restricted class of initial data, the solution $u$ exhibits a sufficient amount of regularity 
to rule out the possibility of finite time blow-up in the usual sense.


\section{The critical case and the case of full off-axis dispersion}\label{sec:crit}

In this section, we shall treat the two ``extreme" cases and consequently prove Theorems \ref{thm:crit} and \ref{thm:full}.

\subsection{Partial off-axis dispersion with critical nonlinearity} 

In the case of partial off-axis dispersion with critical nonlinearity, i.e., $\sigma = \frac{2}{d-k}$ and $0\le k\le 2$, we see that the estimate obtained in Lemma \ref{lem:Nlest} no longer yields a positive power 
$\alpha$ of $T$. Hence the fixed point argument employed in the subcritical case breaks down. 
In order to overcome this obstacle, we shall employ the same type of arguments as in \cite{CW}. 

To this end, we first note 
that a particular admissible pair $(q, r)$ is obtained for $$q=r=\frac{2(d-k+2)}{d-k}$$ and introduce the 
following mixed space for any $I\subset \R_t$:
\begin{equation*}\label{eq:XI}
W(I)=L^{\frac{2(d-k+2)}{d-k}}\Big(I \times \R_x^{d-k};H^{-\frac{d-k}{d-k+2}}(\R_y^k)\Big).
\end{equation*}
Then, we have the following local well-posedness result for $v$, which directly yields Theorem \ref{thm:crit} for $u$ via $v= P_{\eps}^{1/2}u$. 

\begin{proposition}\label{prop:LWPcrit}
Let $d-k > 0$ with $k\le 2$, and $\sigma=\frac{2}{d-k}$. Then for any $v_0\in L^2(\R^d)$, there exist times $0<T_{\rm max}, T_{\rm min}\le\infty$ and 
a unique maximal solution $v\in C((-T_{\rm min}, T_{\rm max});L^2(\R^d))\cap W(I)$ to 
\eqref{PNLS}, where $I$ denotes any closed time interval $I\subset (-T_{\rm min}, T_{\rm max})$. 
Furthermore, $T_{\rm max}<\infty$ if and only if
\begin{equation}\label{eq:bua}
\|v\|_{W([0, T_{\rm max}))}=\infty,
\end{equation}
and analogously for $T_{\rm min}$. Finally, if $\|v_0\|_{L^2}$ is sufficiently small, then the solution is global.
\end{proposition}

Note that here the maximal existence time depends not only on the size of the initial datum but rather on the whole profile of the solution, or more precisely on the $W(I)$-norm of $v$.
\begin{proof}
We shall only give a sketch of the proof for $t\ge 0$, since our arguments follow along the same lines as those in \cite[Section 3]{CW}; see also \cite[Chapter 4.7]{Cz}.

Firstly, given a $T>0$, we claim that by choosing $\delta>0$ sufficiently small and such that 
\begin{equation}\label{eq:small}
\|S_\eps(\cdot)v_0\|_{W([0, T])}<\delta,
\end{equation}
we obtain a unique solution $v\in C([0, T];L^2(\R^d))\cap W([0, T])$ to \eqref{PNLS}. 
Indeed, under assumption \eqref{eq:small}, the operator $v\mapsto \Phi(v)$, defined by \eqref{intPNLS} with $\sigma=\frac{2}{d-k}$, admits a unique fixed point in 
\begin{equation*}
Z_{T, \delta}=\{v\in W([0, T]) \;\textrm{s.t.}\;\|v\|_{W([0, T])}<2\delta\}.
\end{equation*}
As in Proposition \ref{prop:GWPL2}, by means of the Strichartz estimates one can then show that 
$v\in L^q_tL^r_xH^{-\delta(r)}_y(0, T)$ for every admissible pair $(q,r)$. Moreover, since the solution $v$ satisfies the integral equation 
\eqref{intPNLS}, we also infer $v\in C([0, T];L^2(\R^d))$.

To see that $\Phi(v)$ has a fixed point, we use \eqref{StrNL} with $\gamma=\rho=\frac{2(d-k+2)}{d-k}$ and \eqref{eq:small}, to obtain 
\[
\|\Phi(v)\|_{W([0, T])}\le\delta
+C_{\eps}\|v\|_{W([0, T])}^{\frac{4+d-k}{d-k}}.
\]
Since $\frac{4+d-k}{d-k}>1$, choosing $\delta$ small enough guarantees that $\Phi: Z_{T, \delta}\to Z_{T, \delta}$. 
Next, Lemma \ref{lem:Nlest} implies the estimate
\begin{align}\label{eq:critdiff}
\|\Phi(v)-\Phi(w)\|_{W([0, T])} &\le C_{\eps}\left(\|v\|_{W([0, T])}^{\frac{4}{d-k}}+\|w\|_{W([0, T])}^{\frac{4}{d-k}}\right)
\|v-w\|_{W([0, T])},
\end{align}
where $C_{\eps}$ is independent of $T$.
Here, the fact that $\frac{4}{d-k}>0$ and $\delta>0$ is sufficiently small (independent of $v_{0}$ and $T$) 
implies that $v\mapsto \Phi(v)$ is a contraction on $Z_{T,\delta}$. That this choice of $\delta>0$ is always possible follows from our Strichartz estimate and from
\begin{equation}\label{estimate}
\|S_\eps(t)v_0\|_{W([0, T])}\xrightarrow{T \to 0} 0. 
\end{equation}
Consequently, for $T>0$ small enough, 
assumption \eqref{eq:small} is satisfied, yielding a unique local-in-time solution $v(t,\cdot)$ for $t\in[0,T]$.
 
By a similar argument as in Proposition \ref{prop:GWPL2} (see also \cite{Cz, CW}), one can prove uniqueness by letting $v=\Phi(v),w=\Phi(w) \in W([0, T])$ and having in mind that 
$$
\left(\|v\|_{W([0, T])}^{\frac{4}{d-k}}+\|w\|_{W([0, T])}^{\frac{4}{d-k}}\right)  \xrightarrow{T \to 0} 0.
$$
From \eqref{eq:critdiff}, we thus conclude that $v=w$ for $T>0$ sufficiently small. 
We can then iterate this argument to find a maximal existence time $0<T_{\rm max}\le \infty$ for which the unique solution exists for every admissible pair $(q, r)$. 

Next, we shall prove the blow-up alternative \eqref{eq:bua} by contradiction. Namely, 
let $T_{\rm max}<\infty$ and let us assume that $\|v\|_{W([0, T_{\rm max}))}<\infty$. Let $t\in[0, T_{\rm max})$, then for any $s\in[0, T_{\rm max}-t)$ we write in view of \eqref{intPNLS} that
\begin{equation*}
S_\eps(s)v(t)=v(t+s)-\mathcal N(v(t+\cdot))(s).
\end{equation*}
Applying again Lemma \ref{lem:Nlest} we can estimate
\begin{equation*}
\|S_\eps(\cdot)v(t)\|_{W([0, T_{max}-t))}\le\|v\|_{W([t, T_{max}))}+C_{\eps}\|v\|_{W([t, T_{max}))}^{\frac{4+d-k}{d-k}}
\end{equation*}
and thus, for $t$ sufficiently close to $T_{\rm max}$, we have
\begin{equation*}
\|S_\eps(\cdot)v(t)\|_{W([0, T_{\rm max}-t))}<\delta.
\end{equation*}
This implies we can extend the solution after the time $T_{\rm max}$, contradicting its maximality. 

Finally, in order to conclude global existence of small solutions, we note that, by a global-in-time Strichartz estimate, 
\[
\|S_\eps(\cdot)v_0\|_{W(\R)}\le C_{1}\|v_0\|_{L^2}.
\] 
This implies that if $\|v_0\|_{L^2}$ is small enough depending on $\delta>0$,  
we have $$\|S_\eps(\cdot)v_0\|_{W(\R)}<\delta.$$ 
Hence, assumption \eqref{eq:small} is satisfied for all $T\in \R$ and the same continuity argument as before allows one to repeat the contraction argument with $T=\pm \infty$, cf. \cite[Remark 4.7.5]{Cz}. 
In summary, this yields a unique global solution $v(t, \cdot)\in L^2(\R^d)$ for sufficiently small initial data.
\end{proof}


\subsection{The case of full off-axis dispersion}

We finally turn to the case of full off-axis dispersion, i.e., $d=k$. It is clear from our admissibility condition in Definition \ref{def:adm}, that in this case, 
we cannot expect to have any Strichartz estimates (see also \cite{Car} for more details). 
We thus have to resort to a more basic fixed point argument to prove the following result.

\begin{lemma}\label{lem:deqk}
Let $d=k\ge 1$ and $\sigma \le \frac{2}{(d-2)_{+}}$.  Then, for 
any $v_{0} \in L^{2}(\R^{d})$, there exists a unique global solution $v \in  C(\R_t,L^{2}(\R^{d}))$ to \eqref{PNLS}, depending continuously on the initial data and satisfying
\[
\| v(t, \cdot) \|^{2}_{L^{2}}  = \| v_{0} \|^{2}_{L^{2}} \quad \forall \, t\in \R.
\]
\end{lemma}

\begin{proof}
To prove this result it suffices to show that $v\mapsto \Phi(v)$ is a contraction on 
$$
Y_{T,M} = \{ v \in  L^{\infty}([0,T];L^{2}(\R^{d})) : 
\|v \|_{L_{t}^{\infty}L^{2}} \le M \}.
$$
Let $v, v'\in  Y_{T,M}$, and recall that $S_\eps(t)$ is unitary on $L^2$. Using Minkowski's inequality and the scaling argument \eqref{act2} then yields
\begin{align*}
\|\mathcal{N}(v)(t) - \mathcal{N}(v')(t)\|_{L^{2}} \le \eps^{-1}\int_{0}^{t}\|g(P_{\eps}^{-1/2}v)-g(P_{\eps}^{-1/2}v')\|_{H^{-\frac{d\sigma}{2(\sigma +1)}}}(s)\;ds,
\end{align*}
provided $\frac{d\sigma}{2(\sigma +1)} \le 1$, i.e., $\sigma \le \frac{2}{(d-2)_{+}}$. 

By a similar embedding strategy as in Lemma \ref{lem:Nlest} one finds
\begin{align*}
\|g(P_{\eps}^{-1/2}v)-g(P_{\eps}^{1/2}v')\|_{H^{-\frac{d\sigma}{2(\sigma +1)}}} &\le (\|P_{\eps}^{-1/2}v\|_{L^{\rho}}^{2\sigma} + \|P_{\eps}^{-1/2}v'\|_{L^{\rho}}^{2\sigma})\| P_{\eps}^{-1/2}(v - v')  \|_{L^{\rho}}\\
&\le \eps^{-(2\sigma+1)}(\|v\|_{L^{2}}^{2\sigma} + \|v'\|_{L^{2}}^{2\sigma})\| v - v'  \|_{L^{2}},
\end{align*}
which consequently implies that
$$
\|\mathcal{N}(v) - \mathcal{N}(v')\|_{L_{t}^{\infty}L^{2}} \le \eps^{-(2\sigma+1)}T(\|v\|_{L^{\infty}_{t}L^{2}}^{2\sigma} + \|v'\|_{L^{\infty}_{t}L^{2}}^{2\sigma})\| v - v'  \|_{L^{\infty}_{t}L^{2}}.
$$
Choosing $T>0$ sufficiently small, Banach's fixed point theorem directly yields a 
local-in-time solution $v\in C([0,T],L^{2}(\R^{d}))$. The conservation property of the $L^2$-norm of $v$ can then be shown 
analogously as in the proof of Proposition \ref{prop:GWPL2}. This consequently allows us to extend the local solution $v$ for all $t\in \R$. 
\end{proof}

This directly yields Theorem \ref{thm:full} for $u$, since in the case of full-off axis dispersion $v= P_{\eps}^{1/2}u\in L^2(\R^d)$ implies $u\in H^1(\R^d)$ for any $\eps>0$. 
In addition, the $L^2$-conservation for $v$ directly yields \eqref{CLM}, whereas \eqref{CLE} is a standard computation, and valid for any $H^1$-solution $u$. Finally, it is 
straightforward to extend the solution to $v(t,\cdot)\in H^s(\R^d)$ for any $s>0$ provided the initial data satisfies $v_0\in H^s(\R^d)$.

\begin{remark}Note that \eqref{CLM} also implies a uniform-in-time bound on the $H^1$-norm of $u(t,\cdot)$ for any $\eps >0$. In turn, this means that both the kinetic and the nonlinear potential energy remain uniformly bounded for all $t\in\R$. 
\end{remark}


\bibliographystyle{amsplain}

\end{document}